\documentclass[smallextended,referee,envcountsect,]{svjour3}
\smartqed

\begin{document}

\title{Solution and Optimal Strategies for a Differential Game of Pursuit-Evasion with Point-Wise Constraints.}
\titlerunning{Solution and Optimal Strategies for a Differential Game}

\author{Jamilu Adamu$^{1,2}$. Abbas Ja'afaru Badakaya$^{2}$. Felix Wallace 
Tarry$^{3}$. Mehdi Salimi$^{3}$}
\authorrunning{ Jamilu, A. et al. }
\institute{Jamilu Adamu  \at
              jamiluadamu88@gmail.com 
           \and
            Abbas Ja'afaru Badakaya \at
              ajbadakaya.mth@buk.edu.ng 
               \and
             Felix Wallace Tarry   \at
              felixw-t2004@hotmail.ca
             \and
             Mehdi Salimi, Corresponding author \at
              mehdi.salimi@kpu.ca
              \and    
 \begin{itemize}
 \item[1] Department of Mathematics, Federal University Gashua, P.M.B. 1005, Gashua, Yobe State, Nigeria.
 \item[2] Department of Mathematical Sciences, Bayero University Kano, P.M.B. 3011, Kano State, Nigeria.
 \item[3] Mathematics Department, Kwantlen Polytechnics University, BC, Canada.
 \end{itemize}}
 \date{Received: date / Accepted: date}

\maketitle

\begin{abstract}
This work study a differential game of pursuit-evasion involving countably many pursuers $p_{1},p_{2},\cdots,p_{m}$ and a single evader $e$ under point-wise (geometric) constraints in the sequence space $l_{2}.$ The pursuers evolve according to specified differential equations of $1^{st}$ order  while the evader follows a $2^{nd}$ order differential equation. The game period is a fixed time interval of length $\theta$ unit of time. The minimum distance between evader and the pursuer at the terminal time denotes the game payoff. The pursuers aim is to minimize the payoff, whereas the evader seeks to maximize it. We obtain the solution of the game, including optimal strategies of the players construction, as well as game value estimation.
\end{abstract}
\keywords{Pursuer \and Evader \and Pursuit-Evasion \and Geometric constraints}

\newpage
\section{Introduction}
\noindent Differential games of pursuit-evasion constitute a special class of problems in dynamical systems that examine strategic interactions between two or more agents with conflicting objectives, each striving to achieve its own goal. The concept was introduced by Isaacs \cite{Isaacs1965}, and numerous fundamental results have since been developed and documented in classical works, including Petrosyan’s monographs, which occupy a central place in the development of pursuit-evasion differential games. In his seminal work Differential Games of Pursuit, later expanded and published in \cite{Petrosyan 1993}, he provided one of the earliest systematic and rigorous treatments of pursuit differential games as an independent class within differential game theory. These contributions established a unified mathematical framework for games governed by systems of ordinary differential equations, formalizing key concepts such as capture conditions, admissible controls, strategies, and payoff functionals. Other notable contributions include those of Samatov \cite{B. Samatov 2023,B. Samatov 2020}, Petrov \cite{Petrov 2022,Petrov 2025}, Lewin \cite{Lewin1994}, Ibragimov et al. \cite{Ibragimov 2010,Ibragimov 2009,Ibragimov 2005,Ibragimov 2015}, Badakaya et al. \cite{Badakaya 2021,Badakaya 2022,Badakaya 2024} and Krasovskii and Subbotin\cite{Krasovskii1974,Krasovskii1988,REF15}, among others. Differential game provides valuable insights into the behavior of agents in conflict situations arising in real-world applications, including reach-avoid games, surveillance and security, economic processes, ecological equilibrium, labor employer relations, disease control in medicine,  and counter insurgency problems.\\

\noindent Pursuit-evasion differential games, which involve estimating game value along with constructing best strategies for the players, are the main focus of this paper. Papers \cite{Ahmad 2019,Adamu 2020,Badakaya 2021,Badakaya 2022,Ibragimov 2010,Ibragimov 2009,Ibragimov 2005,Rilwan 2023} and the references therein provide examples of works addressing these types of problems. In these studies, the dynamics of the two parties are denoted by systems of ordinary differential equations of several order.\\

\noindent For instance, in Badakaya et al. \cite{Badakaya 2021,Badakaya 2022}, the agents’ (pursuer and evader) dynamics  obey $nth$ order differential equations, thereby generalizing the work of Ibragimov and Salimi \cite{Ibragimov 2005,Salimi 2019}, which considers problems where the agents’ motion are considered by $1^{st}$ order differential equations. In contrast, the studies in \cite{Ibragimov 2010,Ibragimov 2009} examine cases in which the players’ move according to $2^{nd}$ order differential equations. In other works (see, for example, \cite{Adamu 2020}), the dynamics of each pursuer are governed by $1^{st}$ order differential equations, while that of the evader obey $2^{nd}$ differential equation.\\

\noindent Players employ admissible control strategies to drive the system from its initial state toward desired terminal outcomes (capture or escape) while satisfying prescribed control constraints. In pursuit-evasion differential games, these constraints are commonly modeled as geometric constraints, which impose pointwise bounds on the instantaneous control magnitude or direction (e.g., $\vert a(t) \vert \leq \rho$ for all $t$) \cite{Ibragimov 2010,Ibragimov 2005}, and integral constraints, which reflect limitations on the total control effort or energy resources over the game horizon (e.g., $\int_{0}^{T}\vert b(t)\vert^{2}dt \leq \sigma^{2}$). Such constraints determine the feasible maneuvering capabilities of each player and strongly influence optimal pursuit and evasion strategies \cite{Ibragimov 2009,Kazimirova 2024}.\\

\noindent The papers by Salimi \cite{Salimi 2019} investigate a differential game involving multiple pursuers and a single evader in the sequence space $ l_{2} $. The dynamics of all players are governed by first-order differential equations, while their control functions are subject to integral constraints. The authors established sufficient conditions for the existence of the value of the game and constructed optimal strategies for the players. A related problem, in which the players’ controls are subject to geometric constraints, was studied in \cite{Ibragimov 2005}.\\

\noindent In contrast, the work in \cite{Ibragimov 2009} investigates a fixed-duration pursuit-evasion differential game involving infinitely many pursuers and a single evader in the space $l_{2}$. The dynamics of the players are described by second-order differential equations, and the control functions of all players are subject to integral constraints. The authors established sufficient conditions for the existence of the value of the game and constructed optimal strategies for the players. A related problem, in which the integral constraints are replaced by geometric constraints, was studied in \cite{Ibragimov 2010}, where analogous results were obtained.\\

\noindent Adamu et al. \cite{Adamu 2020} check a fixed-horizon pursuit-evasion differential game in the sequence space $l_{2}$ under integral constraints. Each pursuer moves according to certain $1^{st}$ order differential equations, whereas the trajectory path followed by evader is governed by a $2^{nd}$ order differential equation.  The payoff functional is defined as the minimum distance between the evader and the pursuer at the terminal time. The authors determined the game value and formulate the best strategies for the players.\\

\noindent Motivated by these developments, this paper investigates a pursuit-evasion differential game with geometric constraints on the players' control functions. The objectives are to derive sufficient conditions for capture, construct optimal admissible strategies that ensure success under resource limitations, and determine the value of the game. By extending and complementing existing results, this study contributes to the theoretical understanding of pursuit-evasion dynamics under geometric constraints and provides insights relevant to practical applications involving motion-related quantities such as speed and acceleration.

\section{Statement of the Problem}
 We consider a differential game problem of pursuit-evasion in which
pursuers $(p_{j}), \ j \in J = \{1,2, \dots, n,  \dots\}$ and evader $(e)$
move according to the following initial value problem:
\begin{equation}\label{D1}
\begin{array}{ll}
\dot{p}_{j} = a_{j}(t), & p_{j}(0) = p_{j0}, j \in J, \\
\ddot{e} = \  b(t), & \dot{e}(0) = e^{1}, \ e(0) =e^{0}, \\
\end{array}
\end{equation}
where $p_{j}, p_{j0}, a_{j},e, e^{0},e^{1}, b \in l_{2},  \
a_{j}=\left(a_{1},a_{2},\cdots \right)  $ stand for the control strategies (functions) of the pursuers $p_{j}$ and $b(t)$ is that of the evader $e$. The space $l_2 $ consists of elements of the form  $g = (g_{1}, g_{2}, \dots)$ such that
 $ \sum_{m=1}^{\infty}g_{m}^{2}< \infty ,$ with scalar product and norm defined by
 $$\langle \rho, \sigma\rangle = \sum\limits_{m =1}^{\infty} \rho_{m}\sigma_{m}, \
  ||g|| =\left(\sum\limits_{m = 1}^{\infty} g_{m}^2 \right)^{1/2},$$
where $\rho, \sigma, g  \in l_2,$ respectively. Let $H(\upsilon, r) := \{p\in l_{2}:||p - \upsilon||\leq r\}$ be a closed ball centered at $\upsilon \in \ell_2$ of radius $r.$ We also define a sphere  having center at $\upsilon$ with radius $r$ by  $S(\upsilon,r) = \{p\in l_{2}:||p - \upsilon|| = r\} .$ 

\begin{definition}
A function $a_j(t) = (a_{1}(t), a_{2}(t),  \dots)$ with
measurable  coordinates such that   
\begin{equation}\label{D2}
||a_{j}(t)|| \leq \alpha_{j}, \ \ \  0 \leq t \leq \theta,
\end{equation}
is called admissible control of the pursuer $p_j,$ where $\alpha_{j}> 0$ defines the maximal speed of the pursuers.
\end{definition}

\begin{definition}
A function $b(t) = (b_{1}(t), b_{2}(t), \dots)$ with measurable coordinates such that
\begin{equation}\label{D3}
||b(t)||
\leq \beta, \ \ \ 0 \leq t \leq  \theta,
\end{equation} is called admissible control of the evader $e,$ where $\beta>0$ defines the maximal acceleration of the evader.
\end{definition}

\begin{definition}
A function $A_{j}(t, a_{0}, e_{0} , b(t))$, $A_{j}: [0, \theta] \times l_{2} \times l_{2} \times l_{2} \rightarrow l_{2},$ is called the pursuers' strategies if the following conditions hold true.
\begin{itemize}
\item[(i)] $A_{j}(t, a_{0}, e_{0} , b(t))$ is measurable.
\item[(ii)] $||A_{j}(t, a_{0}, e_{0} , b(t))|| \leq \alpha_{j}, \ \ \  0 \leq t \leq \theta$
\item[(iii)] The system (\ref{D1}) with $a_{j}(t)=A_{j}(t, a_{0}, e_{0} , b(t))$ has a unique solution $(p_{j}(\cdot), e(\cdot))$ for any admissible control $b = b(t).$ 
\end{itemize}
\end{definition}

\begin{definition}
A pursuers strategies $\bar{A}_{j} , j \in J $ are called optimal if the following condition holds
 $$
 \Psi_{1}(\bar{A}_{1},\bar{A}_{2},\dots,\bar{A}_{n},\dots) = \inf_{A_{1}, A_{2} ,\dots, A_{n}, \dots}\Psi_{1}(A_{1}, A_2, \dots, A_{n}, \dots)
 $$
 where
 $$
 \Psi_{1}(A_{1},A_{2},\dots,A_{n},\dots, )=\sup_{A(\cdot)}\inf_{j\in J}||p_{j}(\theta) - e(\theta)||,
 $$
$A_{j}, j \in J$ are pursuers admissible strategies and $b(\cdot)$ is an admissible control of the evader $e$.
\end{definition}

\begin{definition}
A function $B(t, p_{1},p_{2},\dots, e)$,
$B:[0, \theta] \times l_{2}\times l_{2}\times,\dots, \times l_{2} \rightarrow l_{2},$ is called evader's strategy if the following conditions hold true.
\begin{itemize}
\item[(i)] $B(t, p_{1},p_{2},\dots, e)$ is measurable.
\item[(ii)] $||B(t, p_{1},p_{2},\dots, e)|| \leq \beta, \ \ \  0 \leq t \leq \theta$
\item[(iii)] The system (\ref{D1}) with $b(t)=B(t, p_{1},p_{2},\dots, e)$ has a unique solution $(p_{j}(\cdot), e(\cdot))$ for any admissible control $a(t).$ 
\end{itemize}
\end{definition}

\begin{definition}
The evader's strategy $\bar{B}$ is called optimal if it satisfies the following 
\begin{equation}
\Psi_{2}(\bar{B}) = \sup_{B}\Psi_{2}(B),
\end{equation}
where
\begin{equation}
\Psi_{2}(B)=\inf_{a_{1}(\cdot),\dots,a_{n}(\cdot),\dots}\inf_{j\in J}||p_{j}(\theta)-e(\theta)||,
\end{equation}
$ a_{j}(\cdot)$ are pursuers admissible control and $B$ is any admissible strategy of the evader.
\end{definition}
\noindent One should be wary that the optimal strategy condition is dependent only on the final state of the players, that being their positions at time $\theta$. Therefore there should be no confusion in that a pursuer catching the evader earlier than time $\theta$ does not end the game, and does not change optimality. It is reported in \cite{REF15} that if
$\Psi_{1}(\bar{A}_{1},\bar{A}_{2},\dots,\bar{A}_{n},\dots)= \varphi = \Psi_{2}(\bar{B})$ then the game has a value $\varphi.$\\

\noindent It can be shown that the state equation of the evader in (\ref{D1})
can be mapped to the following first-order differential equation
\begin{equation}\label{REE}
\dot{e} = (\theta - t)b(t),\  e(0)= e_{0} = e^{1}\theta + e^{0} .
\end{equation}
Solving (\ref{REE}) results in
\begin{equation}\label{D69}
e(\theta) = e^0 + \theta e^1 + \int_0^\theta (\theta - t)b(t)dt
\end{equation}
\noindent Now consider the state equation of the evader in (\ref{D1})
$$
\ddot{e} = e(t), \ \dot{e}(0) = e^{1}, \ e(0) =e^{0}
$$
Solving for $e(\theta)$, we have
$$
e(\theta) = e^0 + \theta e^1 + \int_0^\theta \int_0^t b(s) dsdt
$$
Which, by Cauchy's formula for repeated integration, is
\begin{equation}\label{D68}
e(\theta) = e^0 + \theta e^1 + \int_0^\theta (\theta - t)b(t)dt
\end{equation}
The position of the evader in (\ref{D69}) and (\ref{D68}) is the same at time $\theta.$ Therefore any proof on the optimal strategy of one state equation is valid for the other.\\
\noindent \textit{\textbf{Game Description:}}
Here and below, we will call the game in which the players' motion given by equation (\ref{D1}) where the control functions $a_{j}(\cdot)$ and $b(\cdot)$ satisfies the inequalities (\ref{D2}) and (\ref{D3}) respectively, as game $G^{\ast}$.

\noindent \textit{\textbf{Research problem:}} For the game $G^{\ast},$ find the value and construct optimal strategies of the players. \\

\noindent If the evader $e$ and the pursuers $p_{j}, j \in J$ apply their admissible controls $b(t) = (b_{1}(t), b_{2}(t),\dots)$ and $a_{j}(t) =(a_{j1}(t), a_{j2}(t), \dots)$
 respectively, then, according to (\ref{D1}), their corresponding dynamics are given by
\begin{equation}\label{D70}
p_{j}(t)= (p_{j1}(t), p_{j2}(t), \dots), \ \ \
e(t)=(e_{1}(t), e_{2}(t),  \dots).
\end{equation}
where
\begin{equation}\label{PU1}
 p_{j}(t) = p_{j0} + \int_{0}^{t}a_{j}(s)ds,
\end{equation}
and
\begin{equation}\label{EE2}
e(t)  =  e_{0} + \int_{0}^{t}(t - s)b(s)ds.
\end{equation}
\noindent The attainability domain of the pursuer $p_j$
from initial position $p_{j0}$ and until the time $\theta$ is the closed ball
$H_{P_j}(p_{j0}, \alpha_j\theta).$ Indeed, using (\ref{D70}) and (\ref{D2}), we obtain
\begin{eqnarray}
\left|\left|p_{j}(\theta)- p_{j0}\right|\right|
& = & \left|\left|p_{j0}+\int_{0}^{\theta}a_{j}(t)dt-p_{j0}\right|\right| \nonumber\\
& \leq & \int_{0}^{\theta} \Vert a_{j}(t) \Vert dt \leq  \alpha_{j}\theta. \nonumber
\end{eqnarray}

\noindent On the other hand, if $\bar{p} \in H_{P_j}(p_{j0}, \alpha_j\theta)$ and
$ a_{j}(t) :=\frac{\bar{p}-p_{j0}}{\theta},$ (where $a_{j}(t)$ can be shown to be admissible), then we have
\begin{eqnarray}
p_{j}(\theta)& = & p_{j0}+\int_{0}^{\theta}\frac{\bar{p}-p_{j0}}{\theta}dt =\bar{p}. \nonumber
\end{eqnarray}
In a similar way with $b(t) := \frac{2(\bar{e} - e_{0})}{\theta^{2}},$ we can show that the ball $H_E(e_0, \beta\frac{\theta^2}{2})$
is the attainability domain of the evader $e.$

\section{Auxiliary Game }
\noindent This section is concerned with the study of game $G^{\ast}$ but with one pursuer $p$ and one evader $e.$ In this version of the game the players move according to the following equations:
 \begin{equation}\label{AGP}
\begin{array}{ll}
\dot{p}(t) = a(t), &  p(0) = p_{0}\\
\dot{e}(t) = (\theta - t)b(t), & e(0)= e_{0} = e^{1}\theta + e^{0} .
\end{array}
\end{equation}
The goal of the pursuer is to attain the equality $p(\theta) = e(\theta)$ and that of the evader is contrary. The question here is, under what conditions can the pursuer achieve its target? To answer this question we define the following set in the space $l_{2}$ \\
Let 
\begin{equation}\label{P2}
\Delta = \left\{\gamma \in l_2: 2\langle e_{0} - p_{0}, \gamma\rangle
\leq \theta^2\left( \alpha^2 - \frac{\beta^2\theta^2}{4}\right)  + ||e_{0}||^2 - ||p_{0}||^2   \right\}.
\end{equation}

\noindent \textit{\textbf{Construction of the pursuer's strategy}}\\
 Let $\theta >0$ be a fixed constant, and assume that at the initial time $t=0,$ the pursuer is aware of the initial parameters $p_{0}, e_{0}$ and $b(t).$\\
Let the pursuers' strategy be defined as follows;
\begin{equation}\label{J5}
a(t) = \frac{e_{0} - p_{0}}{\theta} + 
 \int_{0}^{\theta} \frac{(\theta-t)b(t)}{\theta}dt, \ \ 0 \leq t \leq \theta
\end{equation}
 
\noindent The following statement provide sufficient condition for the pursuer to achieve its aim:

\begin{lemma}\label{PP2}
Let $e(\theta)\in \Delta,$ then the strategy (\ref{J5}) guarantees both the inequality $\Vert a(t) \Vert \leq \alpha$ for all $t \in \left[0, \theta \right] $ and the achievement of the equality $p(\theta)=e(\theta)$ on the time interval $\left[0, \theta \right] $, in the game $G^{\ast}$.
\end{lemma}
\begin{proof}
Suppose the assumption of the lemma holds, that is $e(\theta)\in \Delta$, then we have
\begin{equation}\label{B53}
2\langle e_{0}-p_{0},e(\theta)\rangle\leq \theta^2\left( \alpha^2 - \frac{\beta^2\theta^2}{4}\right)  + ||e_{0}||^2 - ||p_{0}||^2.
\end{equation}
According to the inequality (\ref{B53}) and using equation (\ref{EE2}), yield
\begin{equation}\label{J66}
2\left\langle e_{0}-p_{0},\int_{0}^{\theta}(\theta-t)b(t)dt\right\rangle \leq \theta^{2}\left(\alpha^{2}-\frac{\beta^{2}\theta^{2}}{4}\right)-||e_{0}-p_{0}||^{2}.
\end{equation}
If the pursuer $p$ applies the strategy (\ref{J5}), and using (\ref{J66}), we have
\begin{eqnarray}
||a(t)||^{2} &=& \left|\left|\frac{e_{0}-p_{0}}{\theta}+\int_{0}^{\theta}\frac{(\theta-r)b(r)}{\theta}dr\right|\right|^{2} \nonumber\\
&=& \left|\left|\frac{e_{0}-p_{0}}{\theta}\right|\right|^{2}+2\left\langle\frac{e_{0}-p_{0}}{\theta},\int_{0}^{\theta}\frac{(\theta-r)b(r)}{\theta}dr\right\rangle
+\left|\left|\int_{0}^{\theta}\frac{(\theta-r)b(r)}{\theta}dr\right|\right|^{2} \nonumber\\
&=& \frac{||e_{0}-p_{0}||^{2}}{\theta^{2}}+\frac{2}{\theta^{2}}\left\langle e_{0}-p_{0},\int_{0}^{\theta}(\theta-r)b(r)dr\right\rangle +\frac{1}{\theta^{2}}\left( \int_{0}^{\theta}(\theta-r)||b(r)||dr\right) ^{2} \nonumber\\
&\leq & \frac{||e_{0}-p_{0}||^{2}}{\theta^2}+\frac{1}{\theta^2}\left(\theta^{2}\left( \alpha^2 - \frac{\beta^2\theta^2}{4}\right) - ||e_{0}-p_{0}||^{2}\right)+\frac{\beta^{2}\theta^{2}}{4} \nonumber\\ 
& \leq & \alpha^{2}. \nonumber
\end{eqnarray}
Thus, the strategy (\ref{J5}) is admissible. Next we show that $p(\theta)=e(\theta)$. Indeed, using (\ref{J5}), we have
\begin{eqnarray}
p(\theta) &=& p_{0}+\int_{0}^{\theta}\left(\frac{e_{0}-p_{0}}{\theta}+\int_{0}^{\theta}\frac{(\theta- r)b(r)}{\theta}dr\right)dr  \nonumber\\
&=& p_{0}+\int_{0}^{\theta}\left(\frac{e_{0}- p_{0}}{\theta}\right)dr + \frac{1}{\theta}\int_{0}^{\theta}\left( \int_{0}^{\theta}(\theta - r)b(r)dr\right)dr  \nonumber\\
&=& e_{0} + \frac{1}{\theta}\int_{0}^{\theta}\left(e(\theta)-e_{0}\right)dr \nonumber\\
&=& e(\theta). \nonumber
\end{eqnarray}
\end{proof}

\section{Solution of the game $G^{\ast}$}
\noindent Here, we provide the solutions to the research questions. Firstly, we cite a useful lemma.
\begin{lemma}\label{lem3}\cite {Ibragimov 2005,Ibragimov 2015} 
Let $H(e_{0},r), \ H(p_{j0},R_{j}), e_{0}\neq p_{j0}, \ j\in J,$ be infinitely many balls. Let $q_{0}\neq 0\in l_{2}$ such that  $\left\langle e_{0}-p_{j0} ,  q_{0} \right\rangle \geq 0 , $ for all $j\in J.$ Define the set
\begin{equation} \label{value}
\displaystyle X_{j}  = 
 \left\{z \in l_{2}: 2\left\langle e_{0}-p_{j0},z
\right\rangle \leq R_{j}^2 - r^2 +\left\|e_{0}\right\|^{2}-\left\|p_{j0}\right\|^{2}\right\},
\end{equation}
then we have the following
\begin{itemize}
\item[1.] If $ \displaystyle H(e_{0},r) \subset \bigcup_{j \in J}H(p_{j0},R_{j}),$ then $
\displaystyle H(e_{0},r) \subset \bigcup_{j \in J}X_{j}.$
\item[2.] If $\varepsilon \in (0,R_{0})$ and $\inf_{j\in J}R_{j}= R_{0} > 0,$ then the set $\bigcup_{j\in J}H(p_{j0},R_{j}-\varepsilon)$ does not contain the ball $H(e_{0},r)$ and there exist a point $\bar{e}\in S(e_{0},r)$ such that $\left|\left|\bar{e}-p_{j0}\right|\right|\geq R_{j}$ for all $j\in J$.
\end{itemize} 
\end{lemma}

\setcounter{theorem}{0}
\begin{theorem}\label{MR}
Let $\gamma \neq 0 \in l_{2}$ such that $\langle e_{0} - p_{j0}, \gamma \rangle \geq 0,$ for all $j \in J$, then the number 
\begin{equation}\label{GV}
\varphi = \inf\left\{\epsilon \geq 0: H\left(e_{0},\beta\frac{\theta^{2}}{2}\right)
\subset\bigcup_{j\in J}H\left(p_{j0},\alpha_{j}\theta + \epsilon\right)\right\}
\end{equation}
is the value of the game $G^{\ast}$.
\end{theorem}

\begin{proof}
Let $g_{j}$,  $j\in J$ be dummy pursuers whose motion described by
\begin{equation}
\dot{g}_{j}(t)= \bar{a}_{j}(t) , \ \ g_{j}(0)= p_{j0}, \ \ \ j \in J,
\end{equation}
and the control function $\bar{a}_{j}(\cdot)$ is such that
\begin{equation} \label{Dmmy}
\left|\left|\bar{a}_{j}(t)\right|\right| \leq \bar{\alpha}_{j}
 = \alpha_{j}+\frac{\varphi}{\theta}, \  \ \ 0 \leq t \leq \theta.
\end{equation}
It can easily be shown that the set of all points reachable by $g_{j}$ from the initial state  $p_{j0}$ until the time $\theta$ is the ball $ H(p_{j0},\bar{\alpha}_{j}\theta).$ \\

\noindent Let  $g_{j}$ uses the strategy defined by
\begin{equation}\label{DPS}
\bar{a}_{j}(t) =  
\frac{e_{0} - p_{j0}}{\theta} + \int_{0}^{\theta}\frac{(\theta - t)}{\theta}b(t)dt, \ \  0 \le t \le \theta
\end{equation}
\noindent Using (\ref{DPS}), we construct the functions $A_{j}(\cdot)$ to be the strategies of the pursuers $p_j, j \in J$  as follows:
\begin{equation}\label{RPS1}
A_{j}(t)= \frac{\alpha_{j}}{\bar{\alpha}_{j}}\bar{a}_{j}(t).
\end{equation}
Thus, the strategy (\ref{RPS1}) is admissibility. Indeed, using (\ref{Dmmy}), we get  
\begin{eqnarray}
\Vert A_{j}(t) \Vert &=& \left\Vert \frac{\alpha_{j}}{\bar{\alpha}_{j}}\bar{a}_{j}(t)  \right\Vert \nonumber\\
&=& \frac{\alpha_{j}}{\bar{\alpha}_{j}}\left\Vert \bar{a}_{j}(t)  \right\Vert \leq  \alpha_{j}. 
\end{eqnarray}
\noindent In accordance with what is presented for the proof we claim that the
number $\varphi$ defined in (\ref{GV}) is the value of $G^{\ast}$.
To prove this claim we show that following:

\begin{equation}\label{GVI}
\sup_{v(\cdot)}\inf_{j\in J}\left|\left|e(\theta)- p_{j}(\theta)\right|\right|\leq
\varphi \leq \inf_{a_{1}(\cdot),\dots,a_{m}(\cdot),\dots}\inf_{j\in J}\left|\left|e(\theta)- p_{j}(\theta)\right|\right|.
\end{equation}

\noindent Firstly, we prove the left hand side of the inequality (\ref{GVI}).
Indeed, by definition of $\varphi$ in (\ref{GV}), we have
$$
H\left(e_{0},\beta\frac{\theta^{2}}{2}\right)
\subset\bigcup_{j \in J}H\left(p_{j0},\alpha_{j}\theta + \varphi \right).
$$
Employing lemma (\ref{lem3}) with $r = \beta\frac{\theta^{2}}{2}$ and $R_j = \alpha_j\theta + \varphi, $ we have
$$
H\left(e_{0},\beta\frac{\theta^{2}}{2}\right)
\subset\bigcup_{j \in \bar{J}}X_{j},
$$
where
$$\bar{J}=\left\{j\in J:S\left(e_{0},\beta\frac{\theta^{2}}{2}\right)\cap B_{p_{j}}(p_{j0},\alpha_{j}\theta+ \varphi)\neq\emptyset\right\};$$
$$
X_{j} = \left\{z\in l_{2}:2\langle e_{0}- p_{j0},z\rangle\leq \left(\alpha_{j}\theta + \varphi\right)^{2}-\left(\beta\frac{\theta^{2}}{2}\right)^{2}+\left|\left|e_{0}\right|\right|^{2}
-\left|\left|p_{j0}\right|\right|^{2}\right\}
$$
\noindent In view of this, the point $e(\theta)\in H\left(e_{0},\beta\frac{\theta^{2}}{2}\right)\subset X_{k}$, $k\in \bar{J}.$ Then according to lemma (\ref{PP2}) and  for the  strategy  (\ref{DPS}), we have $g_{k}(\theta)= e(\theta)$ with the control
$\bar{a}_{j}(\cdot)$ satisfying the constraint $\left|\left|\bar{a}_{j}(t)\right|\right|\leq \bar{\alpha}.$ In line with this and if $p_{j}$ use (\ref{RPS1}), we show that
$
\left|\left|e(\theta)- p_{k}(\theta)\right|\right| \leq  \varphi.
$
 Indeed,
 \begin{eqnarray}
 \left|\left|e(\theta)- p_{k}(\theta)\right|\right| 
 &=& \left|\left|g_{k}(\theta)-p_{k}(\theta)\right|\right| \nonumber\\
 &=& \left|\left|p_{k0}+\int_{0}^{\theta}\bar{a}_{k}(t)dt   
 - p_{k0}-\int_{0}^{\theta}a_{k}(t)dt\right|\right|   \nonumber\\
 &=& \left|\left|\int_{0}^{\theta}\bar{a}_{k}(t)dt           
 - \int_{0}^{\theta}\frac{\alpha_{k}}{\bar{\alpha_{k}}}\bar{a}_{k}(t)dt\right|\right| \nonumber\\
 & \leq & \int_{0}^{\theta}\left|\left|\left(1 - \frac{\alpha_{k}}{\bar{\alpha_{k}}}\right)
 \bar{a}_{k}(t)\right|\right|dt =  \left(\frac{\bar{\alpha_{k}}-\alpha_{k}}{\bar{\alpha_{k}}}\right)
 \int_{0}^{\theta}\left|\left|\bar{a}_{k}(t)\right|\right|dt   \nonumber\\
 & \leq &\left(\frac{\bar{\alpha_{k}} - \alpha_{k}}{\bar{\alpha_{k}}}\right)\bar{\alpha_{k}}\theta 
 =  \varphi.  \nonumber
\end{eqnarray}
Then it follows that $$ \inf_{j\in J}\Vert e(\theta)-p_{j}(\theta)\Vert \leq \varphi, $$ as such $$ \sup_{b(\cdot)}\inf_{j\in J}\Vert e(\theta)-p_{j}(\theta)\Vert \leq \varphi .$$
This prove the left hand of (\ref{GVI}). \\

\noindent Secondly, we prove the right hand of (\ref{GVI}). Indeed, this follows trivially If $\varphi = 0,$ for any admissible control of the evader. Let $\varphi > 0 $ and
by (\ref{GV}) and for any $0 < \varepsilon < \varphi,$ the ball
$H\left(e_{0},\beta\frac{\theta^{2}}{2}\right)$ is not contained in the set
$
\bigcup_{j \in J}H\left(p_{j0},\alpha_{j}\theta + \varphi - \varepsilon\right).
$\\

\noindent Thus, by lemma (\ref{lem3}), there exist $\bar{e}\in S\left(e_{0},\beta\frac{\theta^{2}}{2}\right)$, such that
\begin{equation}\label{L1}
\left|\left|\bar{e}- p_{j0}\right|\right|\geq \alpha_{j}\theta + \varphi.
\end{equation}
Moreover, from the fact that the ball $H(p_{j0}, \alpha_{j}\theta )$ is the set of all points reachable by $p_j$,, we have
\begin{equation}\label{L2}
\left|\left|p_{j}(\theta)- p_{j0}\right|\right|\leq \alpha_{j}\theta.
\end{equation}
In view of (\ref{L1}) and (\ref{L2}), we have for all $j \in J$
\begin{eqnarray}
  \left|\left|\bar{e}- p_{j}(\theta)\right|\right| 
 & \geq & \left|\left|\bar{e}- p_{j0}\right|\right| 
  -\left|\left|p_{j0} - p_{j}(\theta)\right|\right|  \nonumber\\
 & \geq & \alpha_{j}\theta+ \varphi - \alpha_{j}\theta   \nonumber\\
 &=& \varphi. 
\end{eqnarray}
The control of the evader defined by
  $$b(t)= \frac{2}{\theta^{2}}(\bar{e}-e_{0}), \ \ 0\leq t\leq \theta,$$
  takes the evader to the point $\overline{e}$ at the time $\theta.$ The proof of this claim  is as follows:
\begin{eqnarray}
  e(\theta) &=& e_{0}+\int_{0}^{\theta}(\theta - s)b(s)ds \nonumber\\
   &=& e_{0}+\int_{0}^{\theta}(\theta - s)\frac{2}{\theta^{2}}(\bar{e}- e_{0})ds \nonumber\\
  &=& e_{0}+\frac{2(\bar{e}- e_{0})}{\theta^{2}}\int_{0}^{\theta}(\theta - s)ds  \nonumber\\
  &=& \bar{e}. \nonumber
\end{eqnarray}
Consequently, we have
$$
\left|\left|e(\theta)- p_{k}(\theta)\right|\right| \geq  \varphi.
$$
 As such 
$$
\inf_{j\in J}\Vert e(\theta)-p_{j}(\theta)\Vert \geq \varphi .
$$
Hence for any $a_{j}(t)=(a_{1}(t),a_{2}(t),\dots)$ admissible, we get
$$
\inf_{a_{j}(\cdot)}\inf_{j\in J}\Vert e(\theta)-p_{j}(\theta)\Vert \geq \varphi .
$$
This proves the right hand of (\ref{GVI}). Hence, the proof the theorem is complete.
\end{proof}

\section{Conclusion}
\noindent The work considered in this research represent a dynamics game involving agents with different dynamics possibilities under point-wise constraints. This paper is devoted to characterizing sufficient conditions for capture, and to constructing optimal player's strategies that guarantee success despite resource limitations. The main results consists of the solution to an auxiliary game and Theorem (\ref{MR}). In the auxiliary game, it is shown that pursuer can achieve the terminal condition $p(\theta) = e(\theta)$ by employing a strategy different from the parallel approach. Theorem (\ref{MR}) provides sufficient conditions under which the game value $\varphi$ exist and can be determined. \\

\noindent The principal contributions of this paper can be summarized as follows:
\begin{enumerate}
\item[(i)] Pursuit strategy (\ref{J5}) for the pursuer is constructed that ensure the successful completion of the pursuit.
\item[(ii)] A closed-form expression (\ref{GV}) for the game $G^{\ast}$ is obtained and it is guaranteed for both players. 
\end{enumerate}
 
\section*{Availability of data and materials}
No data are associated with this manuscript.

\section*{Conflict of interest}
The authors declare that they have no competing interest.



\end{document}